\documentclass{amsart}
\usepackage[utf8]{inputenc}
\usepackage[top=1in, left=1in, right=1in, bottom=1in]{geometry}
\usepackage{color, amsmath,amssymb,mathtools,graphicx, mathrsfs, amsthm}
\usepackage{epsfig}
\usepackage{float}
\usepackage{setspace}
\usepackage{epstopdf}
\usepackage{import}
\usepackage{lineno}
\usepackage{stackrel}
\usepackage[allcolors = blue,colorlinks]{hyperref}
\usepackage[abs]{overpic}
\usepackage[center]{caption}
\usepackage[font=small]{subcaption}
\usepackage{tikz-cd}
\usepackage{enumitem}
\counterwithin{figure}{section}
\counterwithin{table}{section}
\counterwithin{equation}{section}
\usepackage{contour}
\contourlength{.08em}
\usepackage{cite}
\usepackage{marginnote}
\newtheorem{theorem}{Theorem}[section]
\newtheorem{lemma}[theorem]{Lemma}
\newtheorem{definition}[theorem]{Definition}

\newtheorem{corollary}[theorem]{Corollary}
\newtheorem{conjecture}[theorem]{Conjecture}
\newtheorem{problem}[theorem]{Problem}

\title{The Montesinos-Nakanishi 3-move conjecture for links up to 20 crossings}
\author[Bakshi]{Rhea Palak Bakshi}
\address{Department of Mathematics, University of California, Santa Barbara, USA}
\email{rheapalak@math.ucsb.edu $|$ rheapalakbakshi@gmail.com}
\author[Burton]{Benjamin A. Burton}
\address{School of Mathematics and Physics,
The University of Queensland, Australia}
\email{bab@maths.uq.edu.au}
\author[Guo]{Huizheng Guo}
\address{Department of Mathematics, George Washington University}
\email{hguo30@gwmail.gwu.edu}
\author[Ibarra]{Dionne Ibarra}
\address{School of Mathematics, Monash University, VIC 3800, Australia.}
\email{dionne.ibarra@monash.edu}
\author[Montoya-Vega]{Gabriel Montoya-Vega} 
\address[]{Department of Mathematics, The Graduate Center CUNY, USA, and \newline \indent Department of Mathematics, University of Puerto Rico at R\'io Piedras, PR, USA} 
\email[]{gabrielmontoyavega@gmail.com} 
\author[Mukherjee]{Sujoy Mukherjee} 
\address[]{Department of Mathematics, University of Denver, USA} 
\email[]{sujoymukherjee.math@gmail.com | sujoy.mukherjee@du.edu} 
\author[Przytycki]{J\'ozef H. Przytycki}
\address[]{Department of Mathematics, George Washington University and University of Gda\'nsk}
\email[]{przytyck@gwu.edu}

\begin{document}


\begin{abstract}
    Yasutaka Nakanishi formulated the following conjecture in 1981: every link is 3-move equivalent to a trivial link. While the conjecture was proved for several specific cases, it remained an open question for over twenty years. In 2002, Mieczys{\l}aw D{\c a}bkowski and the last author showed that it does not hold, in general. In this article, we prove the Montesinos-Nakanishi $3$-move conjecture for links with up to 19 crossings and, with the exception of six pairwise non-isotopic links including the Chen link and its mirror image, for links with 20 crossings. Our work completely classifies links up 20 crossings modulo $3$-moves. This work includes computational methods, including new code in Regina that generalises pre-existing knot functions to work with links.
    \end{abstract}
\maketitle
\tableofcontents

\section{Introduction}

In 1981, Yasutaka Nakanishi, who at the time was a graduate student at Kobe University, formulated the following intriguing conjecture: every link is 3-move equivalent to a trivial link. A $3$-move is a local move on a link diagram illustrated in Figure \ref{3move}. The conjecture is known as the Montesinos-Nakanishi conjecture\footnote{Montesinos used $3$ moves for a related but different conjecture related to 3-fold dihedral branch coverings of $S^3$ branched along links \cite{Mon}.} and it was proved in many special cases; for instance, Qi Chen proved it for $5$-braids, with one exception of 20 crossings \cite{Che}. In 2002, Mieczys{\l}aw D{\c a}bkowski and the last author, showed that the conjecture does not hold in general \cite{DP1, DP2, DP3}. Moreover, they observed that the Chen link is a counterexample to the conjecture using the novel concept of the Burnside group of a link, which they introduced in their paper. In this note, we outline the proof of the Montesinos-Nakanishi conjecture for links up to 19 crossings and show that it holds for links of 20 crossings, with the exception of six pairwise non-isotopic links including the Chen link and its mirror image. These links are, however, 3-move equivalent to the Chen link. The tools we use are Conway's basic polyhedra (classified up to 20 crossings), the Vogel-Traczyk algorithm, which changes link diagrams to their braid form, analysis of conjugacy classes of $5$-braids, Burnside groups of links, and most importantly, a computer generated analysis of diagrams up to 20 crossings. Our main result is stated in the following theorem. 

\begin{figure}[ht]
$$\vcenter{\hbox{\begin{overpic}[scale = 1]{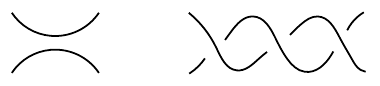}
	\put(52, 18){$\xleftrightarrow[]{3\text{-move }}$}
 \put(20, -5){$D_0$}
\put(130, -5){$D_3$}
\end{overpic} }} $$
\caption{The 3-move operation.}
\label{3move}
\end{figure}

\begin{theorem}[Main Theorem]\label{Main} \ 
\begin{enumerate}
\item[(1)]
The Montesinos-Nakanishi 3-move conjecture holds for links up to 19 crossings.
\item[(2)] A 20 crossings link is either 3-move equivalent to a trivial link or to the Chen link (Figure \ref{Chen5braidTrieste}) possibly with additional trivial components.
\item[(3)] There are six pairwise non-isotopic 20 crossing links without trivial components (all six have a 5-braid representative of 20 crossings) which are not 3-move equivalent to any trivial link; see Equations \ref{eqn1}-\ref{eqn6}. All of them are 3-move equivalent to the Chen link.
\end{enumerate}
\end{theorem}

The diagram in Figure \ref{Chen5braidTrieste}(a) is the Chen link and Figure \ref{Chen5braidTrieste}(b) describes the basic Conway polyhedron which supports it.

\begin{figure}[H]
    \centering
    \begin{subfigure}{.46\textwidth}
    \centering
 \includegraphics[scale=.9]{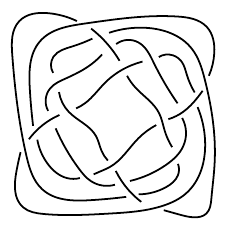}
        \caption{Chen's counterexample to the Montesinos-Nakanishi $3$-move conjecture.}
    \label{ChenLink}
    \end{subfigure}
    \quad
     \begin{subfigure}{.46\textwidth}
     \centering    
    \includegraphics[scale=.9]{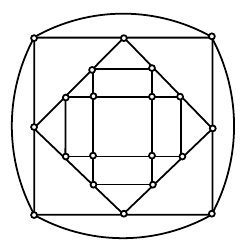}
        \caption{The basic Conway polyhedron supporting the Chen link.}
    \label{BasicConwayPoly}
    \end{subfigure}
    \caption{Chen's 20-crossing counterexample to the Montesinos-Nakanishi $3$-move conjecture, which is the closure of the five braid  $(\sigma_2\sigma_1^{-1}\sigma_2\sigma_3\sigma_4^{-1})^4$.}
    \label{Chen5braidTrieste}
\end{figure}
\section{Method}

Instead of checking the hundreds of millions of links up to 19 crossings \cite{Burton:Knots19}, we first exclude the classes of links for which the conjecture has already been proven. In particular, we exclude $3$-algebraic links \cite{P-Ts} and $5$-braids up to 19 crossings \cite{Cox,Che}. 
\subsection{Conway's basic polyhedra}
An important ingredient of the proof of our main result, Theorem \ref{Main}, is the presentation of link diagrams using basic Conway polyhedra. This method was first used by Kirkman \cite{Kirk, Kirk2} and Tait \cite{Tai}, and formalized by Conway \cite{Con} (see also \cite{JS}).

\begin{definition}\label{BPol} A $4$-regular $4$-edge connected, at least $2$-vertex connected plane graph is called a  
{\it basic Conway polyhedron} or basic polyhedron.
\end{definition}

The connectivity conditions above allow us to keep the number of basic polyhedra small, and they are not restrictive since any link with up to $k$ crossings is supported by some basic polyhedron on no more than $k$ vertices.
The number of basic polyhedra of 20 crossings is equal to $58,782$ and of $19$ crossings is $16,966$. 

\

To be able to analyze all links we should allow decorating vertices of polyhedra by arbitrary 2-algebraic (algebraic in the sense of Conway) 2-tangles. However, since we are considering links up to the 3-move operation, then it is enough to decorate every 4-vertex by a one diagram crossing, see \cite{P-2,P-5}. Notice that we still have $2^k$ possibilities for a $k$-vertex polyhedron.

\subsection{Reductions and rearrangements of Conway's polyhedra}

One method we use that significantly reduces the number of links we need to consider is operations on Conway polyhedra. Some of these operations reduce the number of crossings and some rearrange it while keeping the number of crossings fixed. We say a tangle with $n$ crossings is $3$-move reducible if there exists a sequence of $3$-moves and Reidemeister moves that changes the $n$ crossing tangle into an $m$ crossing tangle where $m<n$.

\

We start by showing that a configuration with two triangles illustrated in Figure~\ref{2triangles-1ab} can be reduced by using the $3$-move operation if the tangle supporting the configuration is non-alternating. Otherwise, the tangle is $3$-move equivalent to an alternating tangle obtained from rotating the $2$-triangles by $\frac{\pi}{3}$. 

\begin{figure}[H]
\centering
$$\vcenter{\hbox{\begin{overpic}[scale=.25]{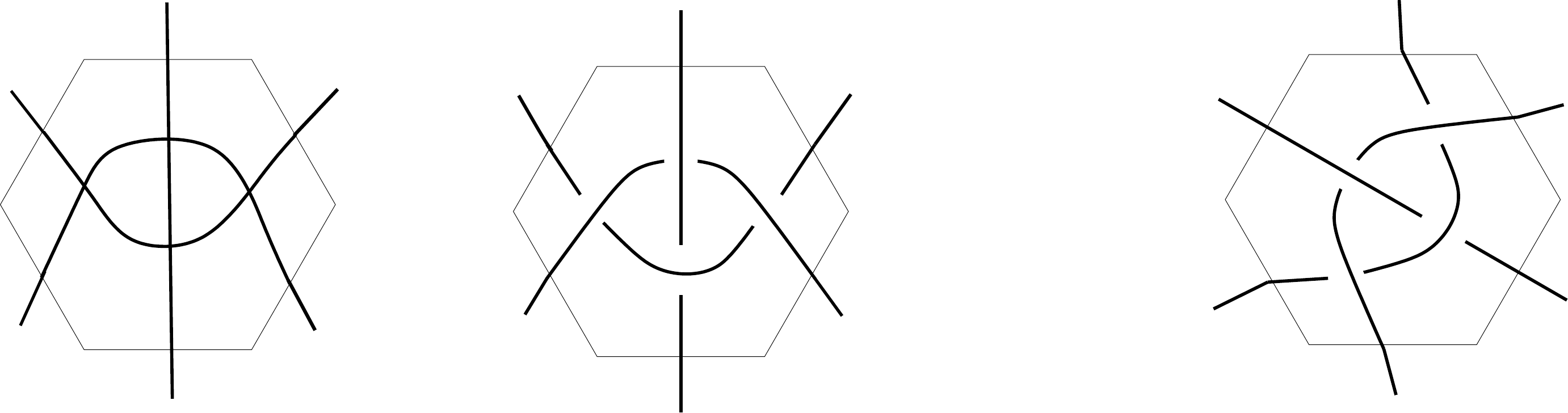}
\put(195, 40){$\xleftrightarrow[\text{ and isotopy }]{3\text{-moves}}$}
\put(0, 10){\textcolor{red}{1}}
\put(-2, 70){\textcolor{red}{2}}
\put(28, 85){\textcolor{red}{3}}
\put(107, 12){\textcolor{red}{1}}
\put(105, 69){\textcolor{red}{2}}
\put(137, 82){\textcolor{red}{3}}
\put(30, -10){(a)}
\put(210, -10){(b)}
\end{overpic}}}.$$
\caption{2-triangles configuration and $\frac{\pi}{3}$ positive (counter-clockwise) rotation.}
\label{2triangles-1ab}
\end{figure}
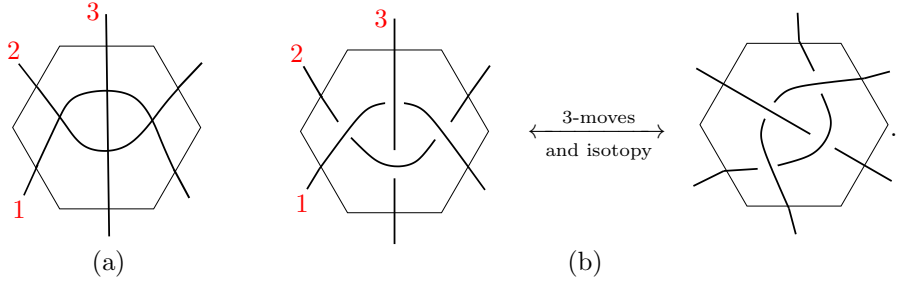

\begin{lemma}\label{redrot}
Consider the 2-triangle 3-tangle configuration of Figure \ref{2triangles-1ab}(a) which is a part of some Conway polyhedron. We draw it in the hexagon to allow better visualisation of a rotation.
\begin{enumerate}
\item[(1)] If the 3-tangle diagram supported by Figure \ref{2triangles-1ab}(a) is not alternating, then it can be reduced by the 3-move operation and isotopy to a diagram of fewer crossings.
\item[(2)] If the diagram supported by Figure \ref{2triangles-1ab}(a) is alternating, then the tangle can be rotated by $\frac{\pi}{3}$ radians using the 3-move operation and isotopy, see Figure \ref{2triangles-1ab}(b).
\end{enumerate}
\end{lemma} 
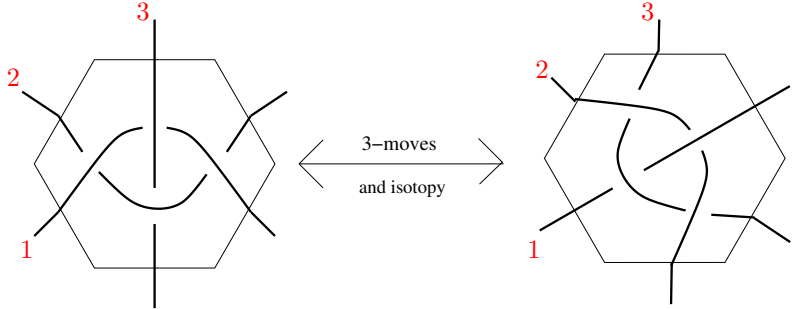
\begin{figure}[H]
\centering
$\vcenter{\hbox{\begin{overpic}[scale=.32]{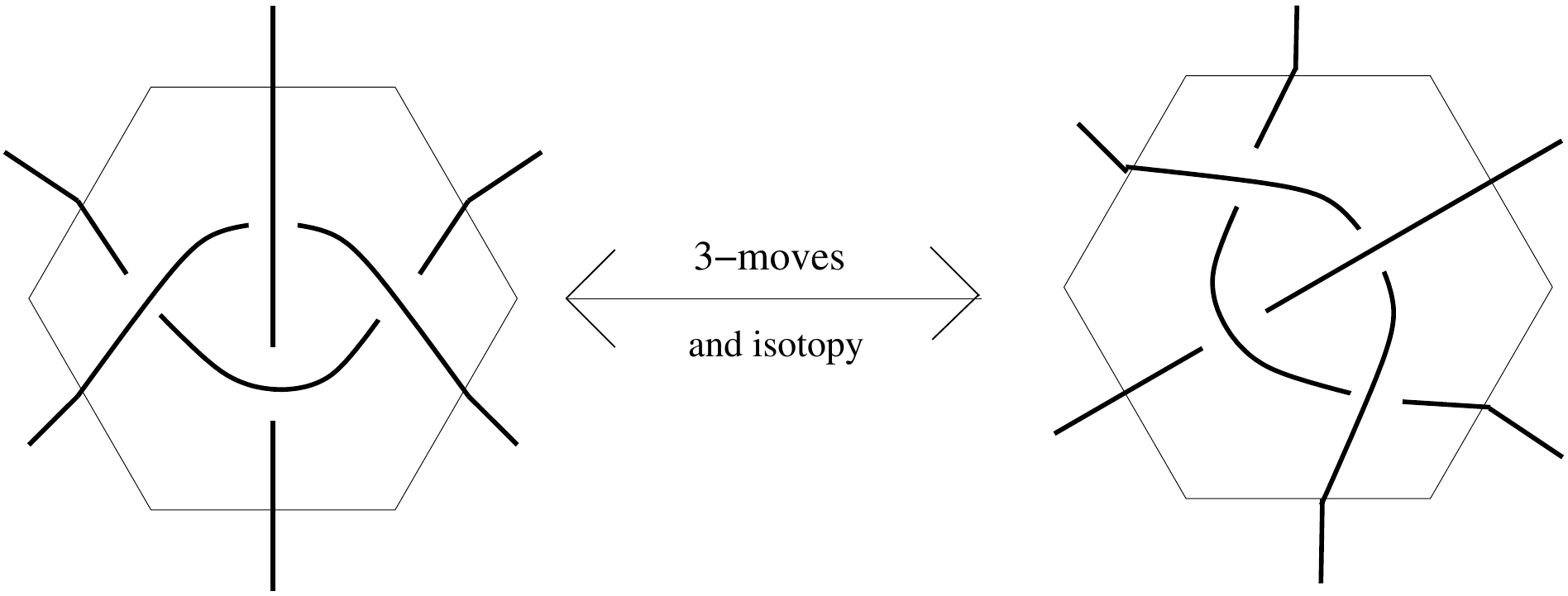}
\put(0,20){\textcolor{red}{$1$}}
\put(192,20){\textcolor{red}{$1$}}
\put(-5,85){\textcolor{red}{$2$}}
\put(195,88){\textcolor{red}{$2$}}
\put(44,110){\textcolor{red}{$3$}}
\put(235,110){\textcolor{red}{$3$}}
\end{overpic}}}$
\caption{Rotation by $\frac{\pi}{3}$ (clockwise direction).}
\label{2triangles-red-rotation1}
\end{figure}
\begin{proof} 
We provide a purely algebraic proof interpreting the first tangle of Figure \ref{2triangles-red-rotation1} as the 3-braid 
$\sigma_1^{-1}\sigma_2\sigma_1^{-1}\sigma_2$ and the second tangle as  
$\sigma_2^{-1}\sigma_1\sigma_2^{-1}\sigma_1$. 
In transforming the first to the second we will use three Reidemeister 3 ($R_3$) moves  and three $3$-move operations, as follows:
$$\sigma_1^{-1}(\sigma_2)\sigma_1^{-1}\sigma_2 \stackrel{3\text{-move}}{=}\sigma_1^{-1}\sigma_2^{-1}(\sigma_2^{-1}\sigma_1^{-1}\sigma_2) \stackrel{R_3}{=} (\sigma_1^{-1}\sigma_2^{-1}\sigma_1)\sigma_2^{-1}\sigma_1^{-1}\stackrel{R_3}{=}\sigma_2\sigma_1^{-1}(\sigma_2^{-1}\sigma_2^{-1})\sigma_1^{-1}\stackrel{3\text{-move}}{=}$$
$$(\sigma_2)\sigma_1^{-1}\sigma_2\sigma_1^{-1}\stackrel{3\text{-move}}{=} \sigma_2^{-1}(\sigma_2^{-1}\sigma_1^{-1}\sigma_2)\sigma_1^{-1}\stackrel{R_3}{=} \sigma_2^{-1}\sigma_1\sigma_2^{-1}(\sigma_1^{-1}\sigma_1^{-1})\stackrel{3\text{-move}}{=} \sigma_2^{-1}\sigma_1\sigma_2^{-1}\sigma_1.$$
\end{proof}

\begin{figure}[H]
\centering
 \begin{subfigure}{.31\textwidth}
       \centering
     \includegraphics[]{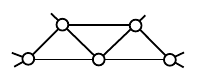}
\caption{3-triangles configuration.}
\label{reduction}
       \end{subfigure} \quad 
\begin{subfigure}{.3\textwidth}
       \centering
\includegraphics[]{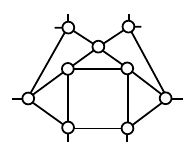}
\caption{4-tangle configuration.}
\label{configurationreducible4}
       \end{subfigure} \quad 
        \begin{subfigure}{.32\textwidth}
       \centering
       $\vcenter{\hbox{\begin{overpic}[]{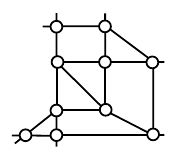}
\end{overpic}}}$
            \caption{4-tangle configuration.}
        \label{fig:Graphreduce9}
       \end{subfigure}
       \caption{Configurations that are 3-move reducible.}
       \label{fig:first3movereducible}
\end{figure}
\begin{figure}[H]
\centering
$\vcenter{\hbox{\begin{overpic}[]{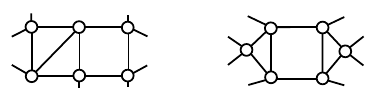}
\put(74, 15){$\xrightarrow[]{3\text{-moves}}$}
\end{overpic}}}.$
\caption{Relocating the square by Reidemeister 3 for alternating $2$-triangles.}
\label{2triangles-square2}
\end{figure}
\begin{figure}[H]
\centering
 \begin{subfigure}{.32\textwidth}
       \centering
$$\vcenter{\hbox{\begin{overpic}[]{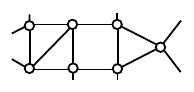}
\end{overpic}}}$$
            \caption{Graph configuration 1.}   \label{fig:Graphreduce0}
       \end{subfigure}
        \begin{subfigure}{.32\textwidth}
       \centering
$$\vcenter{\hbox{\begin{overpic}[]{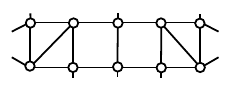}
\end{overpic}}}$$
            \caption{Graph configuration 2.}
        \label{fig:Graphreduce1}
       \end{subfigure}
        \begin{subfigure}{.32\textwidth}
       \centering
    $$\vcenter{\hbox{\begin{overpic}[]{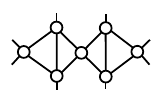}
\end{overpic}}}$$
            \caption{Graph configuration 3.}
        \label{configurationreducible5}
       \end{subfigure}
\caption{$3$-move reducible configurations.}
\label{reducible3move}
\end{figure}

From the lemma we obtain five useful properties.  Corollary \ref{Coro:reducingandmove}(1) is an operation on the tangle, Corollary \ref{Coro:reducingandmove}(2) and (4) are reductions, and Corollary \ref{Coro:reducingandmove}(3) is an operation on the supporting Conway polyhedra. 

\begin{corollary}\label{Coro:reducingandmove}\ 
\begin{enumerate}
\item[(1)] The alternating 2-triangle tangle is 3-move equivalent to its mirror image.
\item[(2)] Any link diagram containing a 3-triangle configuration, as illustrated in Figure \ref{reduction}, 
is 3-move equivalent to one with less crossings (without loss of generality we can assume that every circle has a single crossing).
\item[(3)] The alternating configuration of 2-triangles with the square added is 3-move equivalent to the configuration with the square between two triangles as in Figure \ref{2triangles-square2}.
\item[(4)] Any link diagram containing a tangle supported by a configuration illustrated in Figure \ref{reducible3move}, 
is 3-move equivalent to one with less crossings.
\end{enumerate}
\end{corollary}

\begin{proof} \ 
   \begin{enumerate}
       \item[(1)]  By applying a $\pi$ rotation on the alternating 2-triangle tangle we obtain its mirror image. 
       \item[(2)] We give an alternative proof than that of \cite{Che}. First notice that by Lemma \ref{redrot} any non-alternating tangle supported by the $3$-triangle configuration is $3$-move equivalent to a tangle with less crossings. Now consider an alternating tangle supported by the $3$-triangle configuration. By Corollary ~\ref{Coro:reducingandmove}(1), one choice of alternating $2$-triangle tangles is $3$-move equivalent to its mirror image. However this causes the $3$-triangle tangle to be non-alternating which implies that it is $3$-move equivalent to a tangle with less crossings. An example is illustrated below.
$$\vcenter{\hbox{\begin{overpic}[scale=1]{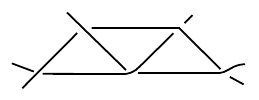}
\end{overpic}}} \xrightarrow[]{\text{Cor. \ref{Coro:reducingandmove}(1)}}
\vcenter{\hbox{\begin{overpic}[scale=1]{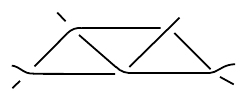}
\end{overpic}}}.$$
\item[(3)] Recall that the alternating tangles supported by the $2$-triangle configuration are $3$-move equivalent to each other so it suffices to consider one alternating case. By Lemma \ref{redrot}, the tangle configuration can be $3$-move changed to the following configuration: 
 $$\vcenter{\hbox{\begin{overpic}[scale=1]{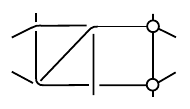}
\end{overpic}}} \xrightarrow[]{3\text{-moves}} \vcenter{\hbox{\begin{overpic}[scale=1]{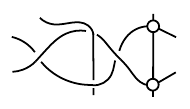}
\end{overpic}}}.$$

Consider a choice of crossing type on one vertex, by a Reidemeister 3 move we obtain a tangle configuration supported by the left configuration illustrated in Figure \ref{2triangles-square2}:

$$\vcenter{\hbox{\begin{overpic}[scale=1]{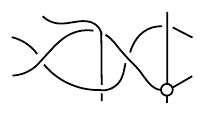}
\end{overpic}}} \sim \vcenter{\hbox{\begin{overpic}[scale=1]{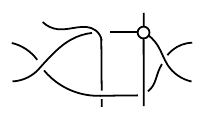}
\end{overpic}}}  .$$

Consider the second choice. After replacing the alternating $2$-triangle tangle by its mirror image and performing a Reidemeister 3 move we again obtain a tangle configuration supported by the left configuration illustrated in Figure \ref{2triangles-square2}:

$$\vcenter{\hbox{\begin{overpic}[scale=1]{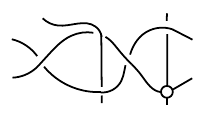}
\end{overpic}}} \xrightarrow[]{\text{Cor. \ref{Coro:reducingandmove}(1)}}  \vcenter{\hbox{\begin{overpic}[scale=1]{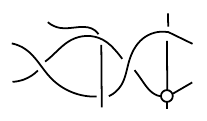}
\end{overpic}}}  \sim \vcenter{\hbox{\begin{overpic}[scale=1]{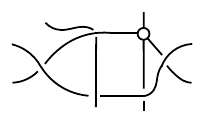}
\end{overpic}}}  .$$
\item[(4)] First note that if the tangle of the link supported by any of the configurations in Figure \ref{reducible3move} contains a non-alternating $2$-triangle tangle then the link is $3$-move equivalent to one with less crossings. We will assume that the $2$-triangle tangles are alternating.

\

For the $7$-vertex configuration illustrated in Figure \ref{reducible3move} left, after applying Corollary \ref{Coro:reducingandmove}(3) we obtain a $7$-vertex configuration with a bigon and bigons can be $3$-move reduced:

$$\vcenter{\hbox{\begin{overpic}[]{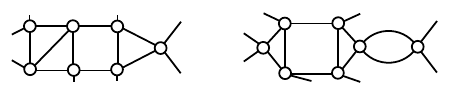}
\put(100, 20){$\rightarrow$}
\end{overpic}}}.$$
For the the $10$-vertex configuration illustrated in Figure \ref{reducible3move} middle after applying Corollary \ref{Coro:reducingandmove}(3), we obtain a $10$-vertex configuration that contains the $3$-triangle configuration illustrated in Figure \ref{reduction}:
$$\vcenter{\hbox{\begin{overpic}[scale=1]{Figures/ConfigReduce3.pdf}
\end{overpic}}} \qquad \rightarrow 
 \qquad \vcenter{\hbox{\begin{overpic}[scale=1]{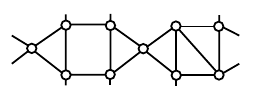}
\end{overpic}}}. $$

By Corollary \ref{Coro:reducingandmove}(1), this configuration is $3$-move reducible.

\

For the the $7$-vertex configuration illustrated in Figure \ref{reducible3move} right after applying Corollary \ref{Coro:reducingandmove}(1) then Corollary \ref{Coro:reducingandmove}(3), we obtain a $7$-vertex configuration that contains a bigon:
$$\vcenter{\hbox{\begin{overpic}[]{Figures/ReducedMoves6big.pdf}
\end{overpic}}} \rightarrow \vcenter{\hbox{\begin{overpic}[]{Figures/ConfigReduce2move2.pdf}
\put(89, 18){$\rightarrow$}
\end{overpic}}}.$$
   \end{enumerate}

\end{proof}

\begin{figure}[H]
\centering
 \begin{subfigure}{.16\textwidth}
       \centering
$$\vcenter{\hbox{\begin{overpic}[scale=1]{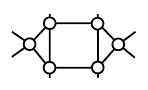}
\end{overpic}}}$$
            \caption{Config 1.}
        \label{fig:Graphmove1}
       \end{subfigure}
 \begin{subfigure}{.16\textwidth}
       \centering
       $$\vcenter{\hbox{\begin{overpic}[scale=1]{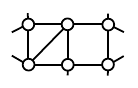}
\end{overpic}}}$$
            \caption{Config 2.}
        \label{fig:Graphmove0}
       \end{subfigure}
        \begin{subfigure}{.16\textwidth}
       \centering
       $$\scalebox{-1}[1]{\includegraphics[scale=1]{Figures/Graphmove0}}$$
            \caption{Config 3.}
        \label{fig:Graphmove2}
       \end{subfigure}
 \begin{subfigure}{.16\textwidth}
       \centering
       $$\vcenter{\hbox{\begin{overpic}[scale=1]{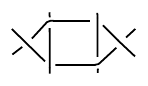}
\end{overpic}}}$$
            \caption{Tangle 1.}
        \label{fig:tangle1}
       \end{subfigure}
        \begin{subfigure}{.16\textwidth}
       \centering
       $$\scalebox{-1}[1]{\includegraphics[scale=1]{Figures/Graphmove2db}}$$
            \caption{Tangle 2.}
        \label{fig:tangle2}
       \end{subfigure}
        \begin{subfigure}{.16\textwidth}
       \centering
$$\vcenter{\hbox{\scalebox{1}[-1]{\includegraphics[scale=1]{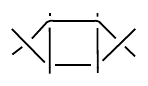}}}}$$
            \caption{Tangle 3.}
        \label{fig:tangle3}
       \end{subfigure}
\caption{Illustration of configurations and tangles related by the 3-move operation.}
\label{leftovertangles}
\end{figure}

\begin{lemma}\label{Lemma:congifmove2}
    The configuration in Figure \ref{fig:Graphmove1} either
    \begin{enumerate}
        \item supports a $6$-crossing tangle that can be reduced to a $5$-crossing tangle by a sequence of Reidemeister moves and $3$-moves, or
        \item supports a 6-crossing tangle that can be changed by a sequence of Reidemeister moves and $3$-moves to a 6-crossing tangle from the configuration illustrated in  Figure \ref{fig:Graphmove0}, or
         \item supports a 6-crossing tangle that can be changed by a sequence of Reidemeister moves and $3$-moves to a 6-crossing tangle from the configuration illustrated in Figure \ref{fig:Graphmove2}, or
         \item is $3$-move equivalent to Figure \ref{fig:tangle1}, \ref{fig:tangle2}, or \ref{fig:tangle3}. 
    \end{enumerate}
\end{lemma}

\begin{proof}
    Consider the following four configurations:
$$\vcenter{\hbox{\begin{overpic}[scale=1]{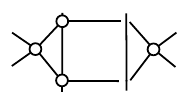}
\end{overpic}}} \quad 
\vcenter{\hbox{\scalebox{-1}[1]{\includegraphics[]{Figures/Graphmove2a1}}}} \quad \vcenter{\hbox{\begin{overpic}[scale=1]{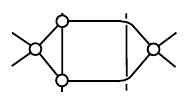}
\end{overpic}}}  \quad \vcenter{\hbox{\scalebox{-1}[1]{\includegraphics[]{Figures/Graphmove2a2}}}}.$$
Notice that each configuration is isotopic to a configuration with $2$-triangles and recall that the non-alternating tangles from the $2$-triangle configuration are $3$-move reducible. Furthermore, the two alternating tangles from each configuration are $3$-move equivalent. Therefore, from the four configurations we only need to consider the following four tangles:

$$\vcenter{\hbox{\scalebox{1}[1]{\includegraphics[]{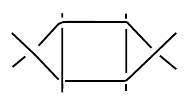}}}}  \quad 
\vcenter{\hbox{\begin{overpic}[scale=1]{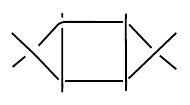}
\end{overpic}}} \quad 
\vcenter{\hbox{\scalebox{-1}[1]{\includegraphics[]{Figures/Graphmove2a3}}}} \quad \vcenter{\hbox{\scalebox{-1}[-1]{\includegraphics[]{Figures/Graphmove2a4}}}}.$$
By Lemma \ref{redrot}, each tangle can be $3$-move changed to a tangle obtained from Figure \ref{fig:Graphmove0} or \ref{fig:Graphmove2}:

$$\vcenter{\hbox{\scalebox{1}[1]{\includegraphics[]{Figures/Graphmove2a4}}}}
\sim
\vcenter{\hbox{\scalebox{1}[1]{\includegraphics[]{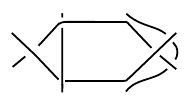}}}}\xrightarrow[]{2-\triangle}
\vcenter{\hbox{\scalebox{1}[1]{\includegraphics[]{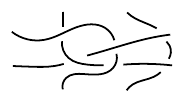}}}}$$

$$\vcenter{\hbox{\begin{overpic}[scale=1]{Figures/Graphmove2a3}
\end{overpic}}}
\sim
\vcenter{\hbox{\begin{overpic}[scale=1]{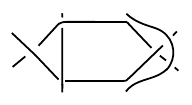}
\end{overpic}}}
\xrightarrow[]{2-\triangle}
\vcenter{\hbox{\begin{overpic}[scale=1]{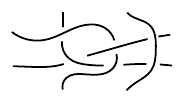}
\end{overpic}}}$$

$$\vcenter{\hbox{\scalebox{-1}[-1]{\includegraphics[]{Figures/Graphmove2a3}}}}
\sim
\vcenter{\hbox{\scalebox{-1}[-1]{\includegraphics[]{Figures/Graphmove2a4e}}}}
\xrightarrow[]{2-\triangle}
\vcenter{\hbox{\scalebox{-1}[-1]{\includegraphics[]{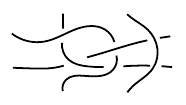}}}}
\xrightarrow[]{2-\triangle}
\vcenter{\hbox{\scalebox{-1}[-1]{\includegraphics[]{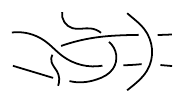}}}}$$

$$\vcenter{\hbox{\scalebox{-1}[-1]{\includegraphics[]{Figures/Graphmove2a4}}}}
\sim
\vcenter{\hbox{\scalebox{-1}[-1]{\includegraphics[]{Figures/Graphmove2a4a}}}}
\xrightarrow[]{2-\triangle}
\vcenter{\hbox{\scalebox{-1}[-1]{\includegraphics[]{Figures/Graphmove2a4b}}}}
\xrightarrow[]{2-\triangle}
\vcenter{\hbox{\scalebox{-1}[-1]{\includegraphics[]{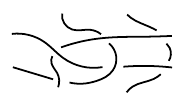}}}}.$$

Now consider the following two configurations:

$$\vcenter{\hbox{\begin{overpic}[scale=1]{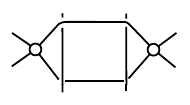}
\end{overpic}}}  \quad  \quad \vcenter{\hbox{\scalebox{1}[-1]{\includegraphics[]{Figures/Graphmove2d}}}}.$$
Two tangles from the two configurations are reducible via Reidemeister moves:

$$\vcenter{\hbox{\begin{overpic}[scale=1]{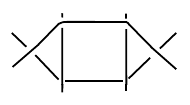}
\end{overpic}}}  \quad  \quad \vcenter{\hbox{\scalebox{1}[-1]{\includegraphics[]{Figures/Graphmove2da2}}}}.$$

Four tangles from the two configurations are reducible by a sequence of Reidemeister moves and $3$-moves:

$$\vcenter{\hbox{\begin{overpic}[scale=1]{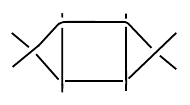}
\end{overpic}}}  \sim
\vcenter{\hbox{\begin{overpic}[scale=1]{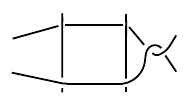}
\end{overpic}}}\xrightarrow[]{3-move}
\vcenter{\hbox{\begin{overpic}[scale=1]{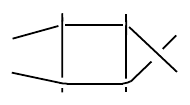}
\end{overpic}}}.$$
$$\vcenter{\hbox{\scalebox{1}[-1]{\includegraphics[]{Figures/Graphmove2da1}}}} \sim
\vcenter{\hbox{\scalebox{1}[-1]{\includegraphics[]{Figures/Graphmove2da3}}}}
\xrightarrow[]{3-move}
\vcenter{\hbox{\scalebox{1}[-1]{\includegraphics[]{Figures/Graphmove2da4}}}}.$$
$$\vcenter{\hbox{\scalebox{-1}[1]{\includegraphics[]{Figures/Graphmove2da1}}}} \sim
\vcenter{\hbox{\scalebox{-1}[1]{\includegraphics[]{Figures/Graphmove2da3}}}}
\xrightarrow[]{3-move}
\vcenter{\hbox{\scalebox{-1}[1]{\includegraphics[]{Figures/Graphmove2da4}}}}.$$
$$\vcenter{\hbox{\scalebox{-1}[-1]{\includegraphics[]{Figures/Graphmove2da1}}}} \sim
\vcenter{\hbox{\scalebox{-1}[-1]{\includegraphics[]{Figures/Graphmove2da3}}}}
\xrightarrow[]{3-move}
\vcenter{\hbox{\scalebox{-1}[-1]{\includegraphics[]{Figures/Graphmove2da4}}}}.$$

The remaining two tangles are $3$-move equivalent to each other and Figure \ref{fig:tangle3}:

$$\vcenter{\hbox{\begin{overpic}[scale=1]{Figures/Graphmove2da5}
\end{overpic}}}  \sim
\vcenter{\hbox{\begin{overpic}[scale=1]{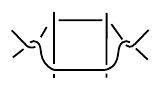}
\end{overpic}}}  
\xrightarrow[]{3-move}
\vcenter{\hbox{\scalebox{1}[-1]{\includegraphics[]{Figures/Graphmove2da5}}}}.$$

Consider the following two configurations:
$$\vcenter{\hbox{\begin{overpic}[scale=1]{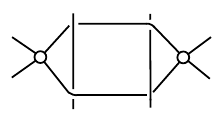}
\end{overpic}}}  \quad \quad 
\vcenter{\hbox{\scalebox{1}[-1]{\includegraphics[]{Figures/Graphmove2dc}}}}.$$

Two tangles from these configurations are $3$-move reducible:

$$\vcenter{\hbox{\begin{overpic}[scale=1]{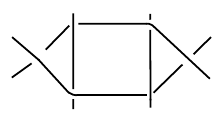}
\end{overpic}}}  \sim
\vcenter{\hbox{\begin{overpic}[scale=1]{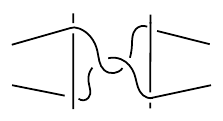}
\end{overpic}}}  
\xrightarrow[]{3-move}
\vcenter{\hbox{\begin{overpic}[scale=1]{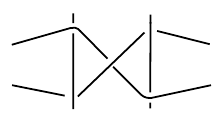}
\end{overpic}}}.$$
$$ \vcenter{\hbox{\scalebox{1}[-1]{\includegraphics[]{Figures/Graphmove2dc1}}}} \sim
\vcenter{\hbox{\scalebox{1}[-1]{\includegraphics[]{Figures/Graphmove2dc2}}}} 
\xrightarrow[]{3-move}
\vcenter{\hbox{\scalebox{1}[-1]{\includegraphics[]{Figures/Graphmove2dc3}}}}.$$
Four tangles from the last two configurations are three move equivalent to a tangle obtained from Figure \ref{fig:Graphmove0} or \ref{fig:Graphmove2}:

$$\vcenter{\hbox{\scalebox{1}[1]{\includegraphics[]{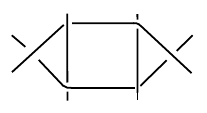}}}} \xrightarrow[]{3-move}
\vcenter{\hbox{\scalebox{1}[1]{\includegraphics[]{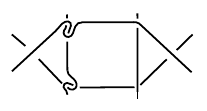}}}}
\sim
\vcenter{\hbox{\scalebox{1}[1]{\includegraphics[]{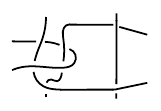}}}}
\xrightarrow[]{3-move}
\vcenter{\hbox{\scalebox{1}[1]{\includegraphics[]{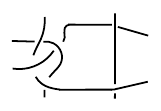}}}}$$ 
$$\vcenter{\hbox{\scalebox{1}[-1]{\includegraphics[]{Figures/Graphmove2dd}}}} \xrightarrow[]{3-move}
\vcenter{\hbox{\scalebox{-1}[-1]{\includegraphics[]{Figures/Graphmove2dd1}}}}
\sim
\vcenter{\hbox{\scalebox{1}[-1]{\includegraphics[]{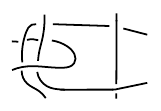}}}}
\xrightarrow[]{3-move}
\vcenter{\hbox{\scalebox{1}[-1]{\includegraphics[]{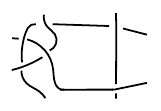}}}}.$$
By symmetry we have,
$$\vcenter{\hbox{\scalebox{-1}[-1]{\includegraphics[]{Figures/Graphmove2dd}}}} 
\xrightarrow[]{3-move}
\vcenter{\hbox{\scalebox{-1}[-1]{\includegraphics[]{Figures/Graphmove2dd6}}}}\text{ and } \vcenter{\hbox{\scalebox{-1}[1]{\includegraphics[]{Figures/Graphmove2dd}}}}
\xrightarrow[]{3-move}
\vcenter{\hbox{\scalebox{-1}[1]{\includegraphics[]{Figures/Graphmove2dd3}}}}.$$

The remaining two tangles are illustrated in Figures \ref{fig:tangle1} and \ref{fig:tangle2}.
\end{proof}

The other more sophisticated and very useful reduction is discussed in the following lemma:

\begin{lemma}\label{prop:reduce3}
    The configuration illustrated in Figure \ref{configurationreducible4} supports $9$-crossing tangles that are 3-move equivalent to a tangle with fewer crossings. 
\end{lemma}

\begin{proof}
   By Lemma \ref{Lemma:congifmove2} we have five cases to consider.
   \begin{enumerate}
       \item The configuration in Figure \ref{configurationreducible4} can support a $9$-crossing tangles that can be reduced by a sequence of Reidemeister moves and $3$-moves.
       
       \item The configuration in Figure \ref{configurationreducible4} can support a $9$-crossing tangle that can be changed (highlighted in gray) by a sequence of Reidemeister moves and $3$-moves to a $9$-crossing tangle from the following configuration:

$$\vcenter{\hbox{\begin{overpic}{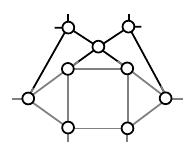}
\end{overpic}}} \xrightarrow[]{3-move} \vcenter{\hbox{\begin{overpic}{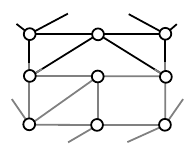}
\end{overpic}}}= \vcenter{\hbox{\begin{overpic}[]{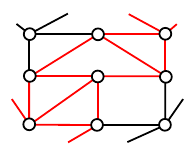}
\end{overpic}}}.$$

       Since this configuration contains Figure \ref{fig:Graphreduce0} as a sub-graph (highlighted in red), then by Corollary \ref{Coro:reducingandmove}(4) the tangles are $3$-move reducible.
       
         \item The configuration in Figure \ref{configurationreducible4} can support a $9$-crossing tangle that can be changed (highlighted in gray) by a sequence of Reidemeister moves and $3$-moves to a $9$-crossing tangle from the following configuration: 

         $$\vcenter{\hbox{\begin{overpic}[]{Figures/ReducedMoves5c1}
\end{overpic}}} \xrightarrow[]{3-move} \vcenter{\hbox{\scalebox{-1}[1]{\begin{overpic}[]{Figures/ReducedMoves5c2}
\end{overpic}}}}=\vcenter{\hbox{\begin{overpic}[]{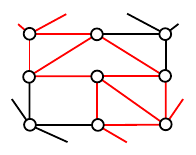}
\end{overpic}}}.$$

         Since this configuration contains Figure \ref{fig:Graphreduce0} as a sub-graph (hightlighted in red), then by Corollary \ref{Coro:reducingandmove}(4) the tangles are $3$-move reducible.
        
         \item The configuration in Figure \ref{configurationreducible4} is $3$-move equivalent to the following $9$-crossing tangle: 
         
         $$\vcenter{\hbox{\begin{overpic}[scale=1]{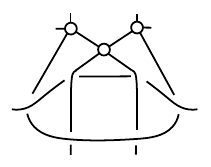}\end{overpic}}}.$$

         The following sequence of $3$-moves and Reidemeister moves show that this tangle-configuration is $3$-move equivalent to a $9$-crossing tangle-configuration containing 2-triangles.

         $$\vcenter{\hbox{\begin{overpic}[scale=1]{Figures/ReducedMoves5T1}\end{overpic}}} \sim \vcenter{\hbox{\begin{overpic}[scale=1]{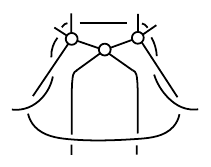}\end{overpic}}} \xrightarrow[]{3-move} \vcenter{\hbox{\begin{overpic}[scale=1]{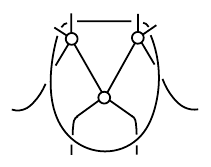}\end{overpic}}}$$
         $$\sim \vcenter{\hbox{\begin{overpic}[scale=1]{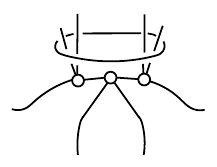}\end{overpic}}}  \xrightarrow[]{3-move} 
 \vcenter{\hbox{\begin{overpic}[scale=1]{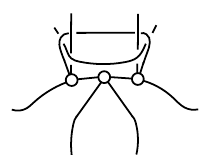}\end{overpic}}} .$$

         This new tangle-configuration comes from a 9-vertex configuration that contains Figure \ref{fig:Graphreduce0} as a sub-graph. 
         $$\vcenter{\hbox{\begin{overpic}[]{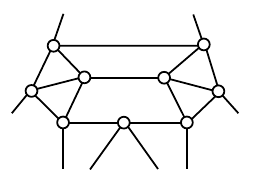}\end{overpic}}} \qquad = \qquad \vcenter{\hbox{\begin{overpic}[]{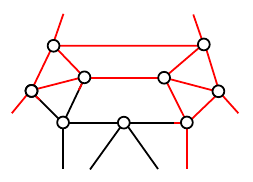}\end{overpic}}}.$$

         Corollary \ref{Coro:reducingandmove}(4), this configuration is $3$-move reducible.

         \item The configuration in Figure \ref{configurationreducible4} is $3$-move equivalent to the following $3$-move reducible $9$-crossing tangle: 
              
          $$\vcenter{\hbox{\begin{overpic}[]{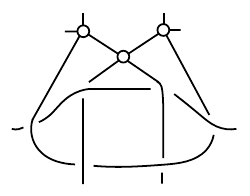}\end{overpic}}}.$$

First consider the tangle-configuration obtained from choosing a crossing for the middle vertex. 
 $$\vcenter{\hbox{{\includegraphics[]{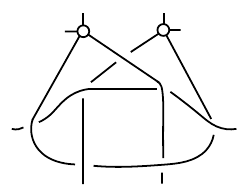}}}} \sim \vcenter{\hbox{{\includegraphics[]{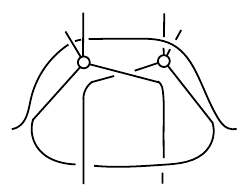}}}}.$$

 The application of a sequence of Reidemeister shows that this new tangle-configuration comes from a $9$-vertex configuration that contains Figure \ref{reduction} as a sub-graph. 

 \begin{equation}\label{eqn:configw3triangle}
\vcenter{\hbox{{\includegraphics[]{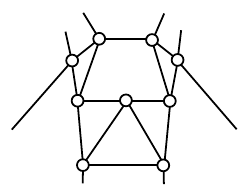}}}} = \vcenter{\hbox{{\includegraphics[]{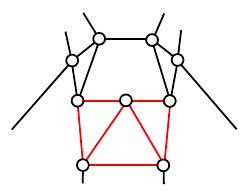}}}}.
 \end{equation}

 By Corollary \ref{Coro:reducingandmove}(2), this configuration is $3$-move reducible. 

\

Now, let the middle vertex be of the second crossing type and choose the crossing type for the right vertex. The following sequence of $3$-moves and Reidemeister moves show that this tangle-configuration is $3$-move equivalent to a $9$-crossing tangle-configuration containing $2$-triangles.

$$\vcenter{\hbox{{\includegraphics[]{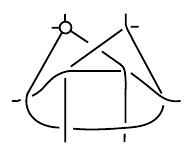}}}} \xrightarrow[]{3-move} \vcenter{\hbox{{\includegraphics[]{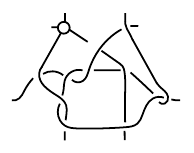}}}} \sim \vcenter{\hbox{{\includegraphics[]{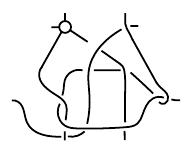}}}}  \sim \vcenter{\hbox{{\includegraphics[]{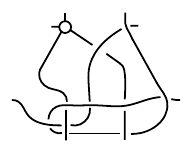}}}} $$

$$\sim \vcenter{\hbox{{\includegraphics[]{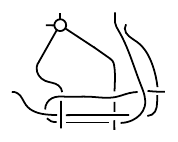}}}} \sim \vcenter{\hbox{{\includegraphics[]{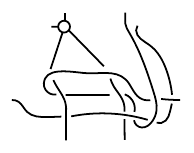}}}} \xrightarrow[]{3-move} \vcenter{\hbox{{\includegraphics[]{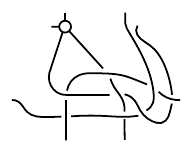}}}}.$$

This new tangle-configuration comes from a $9$-vertex configuration that contains Figure \ref{fig:Graphreduce0} as a sub-graph.
\begin{equation}\label{eqn:configreduce2}
\vcenter{\hbox{{\includegraphics[]{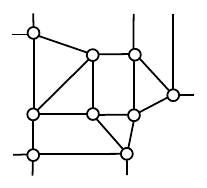}}}} = \vcenter{\hbox{{\includegraphics[]{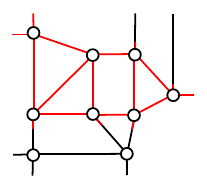}}}}.
\end{equation}

Corollary \ref{Coro:reducingandmove}(4), this configuration is $3$-move reducible.

Now let the right vertex be of the second crossing type. The following sequence of 3-move operations and Reidemeister moves show that this tangle-configuration is $3$-move equivalent to a $9$-crossing tangle-configuration containing $2$-triangles.

$$\vcenter{\hbox{{\includegraphics[]{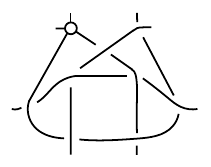}}}} \xrightarrow[]{3-move} \vcenter{\hbox{{\includegraphics[]{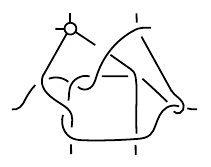}}}} \sim \vcenter{\hbox{{\includegraphics[]{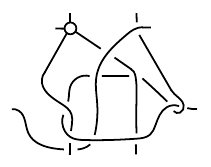}}}}  $$
$$\sim \vcenter{\hbox{{\includegraphics[]{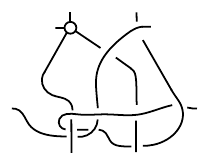}}}} \xrightarrow[]{3-move} 
\vcenter{\hbox{{\includegraphics[]{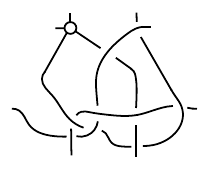}}}} \xrightarrow[]{\text{Cor. }  \ref{Coro:reducingandmove}(1) } \vcenter{\hbox{{\includegraphics[]{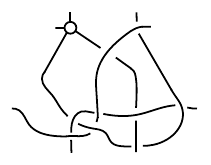}}}} $$
$$\sim \vcenter{\hbox{{\includegraphics[]{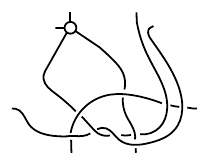}}}}\xrightarrow[]{3-move} 
\vcenter{\hbox{{\includegraphics[]{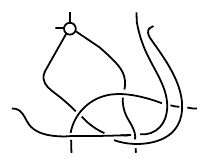}}}} \sim \vcenter{\hbox{{\includegraphics[]{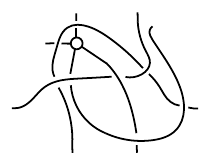}}}}.$$

One choice of crossing type for the remaining vertex yields a reducible 9-crossing tangle with a non-alternating $2$-triangle. Therefore we choose the second crossing type. By applying a sequence of $3$-move operations and Reidemeister moves on this tangle we obtain the following $9$-crossing tangle. 

$$\vcenter{\hbox{{\includegraphics[]{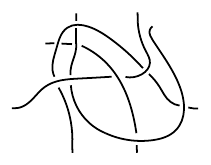}}}} \xrightarrow[]{\text{Cor. }  \ref{Coro:reducingandmove}(1) } \vcenter{\hbox{{\includegraphics[]{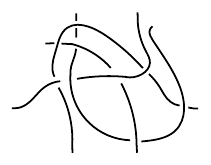}}}} \sim \vcenter{\hbox{{\includegraphics[]{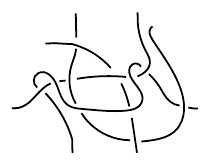}}}} $$

$$\xrightarrow[]{3-move} \vcenter{\hbox{{\includegraphics[]{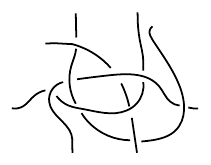}}}} \sim \vcenter{\hbox{{\includegraphics[]{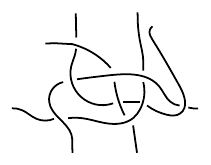}}}}.$$

This new tangle-configuration comes from a $9$-vertex configuration that contains Figure \ref{fig:Graphreduce0} as a sub-graph illustrated in Equation \ref{eqn:configreduce2}. By
Corollary \ref{Coro:reducingandmove}(4), this configuration is $3$-move reducible.

         \item The configuration in Figure \ref{configurationreducible4} is $3$-move equivalent to the following $3$-move reducible $9$-crossing tangle: 
         
          $$\vcenter{\hbox{\scalebox{-1}[1]{\includegraphics[]{Figures/ReducedMoves5T2.pdf}}}}.$$

First consider the tangle-configuration obtained from choosing a crossing for the middle vertex:
$$\vcenter{\hbox{\scalebox{-1}[1]{\includegraphics[]{Figures/ReducedMoves5T2a.pdf}}}} \sim \vcenter{\hbox{\scalebox{-1}[1]{\includegraphics[]{Figures/ReducedMoves5T2a1.pdf}}}}.$$

The application of a sequence of Reidemeister moves show that this new tangle-configuration comes from a $9$-vertex configuration that contains Figure \ref{reduction} as a sub-graph, illustrated in Equation \ref{eqn:configw3triangle}. By Corollary \ref{Coro:reducingandmove}(2), this configuration is $3$-move reducible.

\

Now consider the second crossing type for the middle vertex and make a choice on the crossing type for the right vertex. Similar to case (5) a sequence of $3$-moves and Reidemeister moves can be applied to show that this tangle-configuration is $3$-move equivalent to a $9$-crossing tangle-configuration obtained from the $9$-vertex configuration illustrated in Figure~\ref{eqn:configreduce2}. By Corollary \ref{Coro:reducingandmove}(4), this configuration is $3$-move reducible.

\

Now consider the second crossing type for the right vertex. Similar to case (5) a sequence of $3$-moves and Reidemeister moves can be applied to show that this tangle-configuration is $3$-move equivalent to a $9$-crossing tangle-configuration containing $2$-triangles:

$$\vcenter{\hbox{{\includegraphics[]{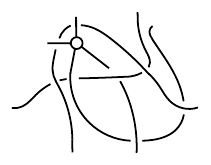}}}}.$$

One choice of crossing type for the remaining vertex yields a reducible 9-crossing tangle with a non-alternating $2$-triangle. Therefore we choose the second crossing type. Again, similar to case (5) a sequence of $3$-moves and Reidemeister moves can be applied to show that this tangle-configuration is $3$-move equivalent to a $9$-crossing tangle-configuration obtained from the $9$-vertex configuration illustrated in Figure~\ref{eqn:configreduce2}. By Corollary \ref{Coro:reducingandmove}(4), this configuration is $3$-move reducible.
   \end{enumerate}
\end{proof}

\begin{corollary}\label{coro:reduce3}
    The configuration illustrated in Figure~\ref{fig:Graphreduce9} supports $10$-crossing tangles that are three move equivalent to a tangle with fewer crossings. 
\end{corollary}

\begin{proof}

$$\vcenter{\hbox{\begin{overpic}[]{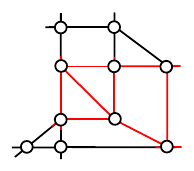}
\end{overpic}}} \xrightarrow[]{\text{Cor. \ref{Coro:reducingandmove}(3)}} \vcenter{\hbox{\begin{overpic}[]{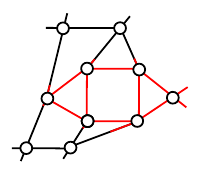}
\end{overpic}}}  = \vcenter{\hbox{\begin{overpic}[]{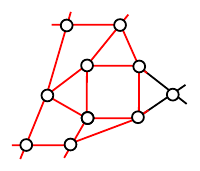}
\end{overpic}}}.$$
\end{proof}

\section{Computational process} \label{section:computation}

We use computational methods to exclude links with properties that imply $3$-move reducibility. This process is done inductively:\\

Suppose all links up to $k$ crossings can be reduced to a trivial link of $n$-components by using a sequence of $3$-moves and Reidemeister moves. In order to prove that all $k+1$ crossing links can be reduced to a trivial link of $n$-components by using a sequence of $3$-moves and Reidemeister moves, it suffices to prove that the links can be reduced to a link with fewer crossings. \\

Since the conjecture was proven for all links up to 11 crossings in \cite{P-2} and links with 12 crossings in \cite{Che}, then we start by creating all basic Conway polyhedra supporting links of 13 through 20 crossings. We obtain the basic Conway polyhedra with no bigons from plantri \cite{plantri}. More precisely, we searched for $4$-edge-connected simple quartic graphs in plantri with the command:

\begin{quote}
\texttt{plantri -adq -c2 \emph{cells}}
\end{quote}

Here \textit{cells} is the number of 2-cells in the the link diagram when drawn on the sphere; by Euler characteristic, this is two more than the number of crossings.

By Corollary \ref{Coro:reducingandmove}(2), any link supported by a $k$-vertex basic Conway polyhedra that contains the $3$-triangle configuration (Figure \ref{reduction}) as a subgraph, can be $3$-move reduced to a link supported by a $m$ basic Conway polyhedra where $m<k$. We implemented into Regina a search for basic Conway polyhedra containing the $3$-triangle configuration as a subgraph and used the code to exclude these basic Conway polyhedra from our list.\\

From each of the remaining basic Conway polyhedra, we build a corresponding collection of links in Regina by replacing each vertex with a crossing in two possible ways (over or under), in a way that avoids obvious non-minimal diagrams. For this we extend the pre-existing knot enumeration code \cite{Burton:Knots19} to include links. We then exhaustively search for and remove all links that are Reidemeister 3 equivalent to a diagram with a bigon, using a breadth-first search through link diagrams. Next, we do a similar search, this time allowing all Reidemeister moves and a small number of extra crossings; this time we only exclude links that are isotopic to a diagram containing a bigon and with the original number of crossings. To implement these breadth-first searches, we extend Regina's knot signatures \cite{Burton:Knots19} to support links, in order to avoid revisiting diagrams that are combinatorially isomorphic to diagrams already seen. \\

We implemented code in Regina to find and exclude all polyhedra supporting links with fewer than five Seifert circles \cite{Che}.  In this step we assign all possible orientations to the components of each supported link: If any of these orientations gives fewer than five Seifert circles then we excluded the link. What remains are links supported by 608 polyhedra ($7 \times$ 16-vertex, $15 \times$ 17-vertex, $56 \times$ 18-vertex, $125 \times$ 19-vertex, and $405 \times$ 20-vertex polyhedra). \\

The graphs were converted to Mathematica's built in Graph symbol \cite{Mathematica}. In Mathematica, we used the \emph{FindIsomorphicSubgraph} function to search for and remove polyhedra containing as a subgraph any configuration illustrated in Figures \ref{configurationreducible4},  \ref{fig:Graphreduce9}, and \ref{reducible3move}. All Conway polyhedra supporting links of 13 through 17 crossings are eliminated from this search and only 11 graphs remained; $1 \times$ 18-vertex, $1 \times$ 19 vertex, and $9 \times$ 20-vertex polyhedra. We built the links from the 11 polyhedra in Regina using the previously explained method. This process produced 12 links with 18 crossings, 31 links with 19 crossings, and 450 links with 20 crossings. \\

By using Regina we found that all links (with properly chosen orientations) supporting the remaining Conway polyhedra have no more than five Seifert circles. We extract the PD codes from Regina then use Vogel-Traczyk's algorithm \cite{Vo} coded in the Knot Theory Mathematica package \cite{knottheory} to obtain the braid words for the links. We observe that the conversion into braid words does not change the number of crossings and the index of each braid word is 5. From \cite{Che}, we may exclude all of the 18 and 19 crossing links.

\

We explore the remaining 450 five-braid diagrams of 20 crossings. We test whether these braid words are conjugate to the Chen link diagram or its mirror image in $C_5=B_5/(\sigma_1^3)$ by using the software GAP \cite{GAP}.\footnote{The 450 links were enumerated and denoted by L1 through L450 in \cite{Guo}.} 

\

We find three $20$ crossing links that are conjugate to the Chen link in $C_5$:

\begin{enumerate}
    \item[L235] PD[X[3, 33, 4, 40], X[4, 13, 5, 12], X[6, 29, 7, 28], 
   X[7, 36, 8, 37], X[10, 39, 11, 38], X[13, 22, 14, 21], 
   X[14, 30, 15, 29], X[16, 1, 9, 8], X[17, 10, 18, 9], 
   X[20, 5, 21, 6], X[23, 32, 24, 31], X[24, 2, 17, 1], 
   X[26, 12, 27, 11], X[27, 20, 28, 19], X[30, 34, 31, 35], 
   X[32, 3, 25, 2], X[33, 23, 34, 22], X[35, 16, 36, 15], 
   X[37, 18, 38, 19], X[39, 25, 40, 26]]
   \item[L270] PD[X[2, 25, 3, 32], X[4, 13, 5, 12], X[6, 22, 7, 21], 
   X[7, 37, 8, 36], X[10, 18, 11, 19], X[13, 26, 14, 27], 
   X[15, 38, 16, 37], X[16, 1, 9, 8], X[17, 32, 18, 31], 
   X[19, 34, 20, 35], X[22, 14, 23, 15], X[24, 2, 17, 1], 
   X[25, 39, 26, 40], X[27, 6, 28, 5], X[28, 21, 29, 20], 
   X[30, 9, 31, 10], X[33, 12, 34, 11], X[35, 29, 36, 30], 
   X[38, 23, 39, 24], X[40, 4, 33, 3]] 
   \item[L298] PD[X[3, 33, 4, 40], X[4, 13, 5, 12], X[6, 22, 7, 21], 
   X[7, 36, 8, 37], X[10, 39, 11, 38], X[13, 31, 14, 30], 
   X[14, 23, 15, 22], X[16, 1, 9, 8], X[19, 12, 20, 11], 
   X[20, 29, 21, 28], X[23, 34, 24, 35], X[24, 2, 17, 1], 
   X[25, 18, 26, 17], X[26, 10, 27, 9], X[29, 5, 30, 6], 
   X[32, 3, 25, 2], X[33, 32, 34, 31], X[35, 16, 36, 15], 
   X[37, 27, 38, 28], X[39, 18, 40, 19]].
\end{enumerate}

  We also find three $20$ crossing links that are conjugate to Chen's mirror image in $C_5$:
\begin{enumerate}
    \item[L176] PD[X[2, 25, 3, 32], X[4, 20, 5, 19], X[6, 36, 7, 35], 
   X[7, 29, 8, 30], X[11, 5, 12, 6], X[13, 26, 14, 27], 
   X[15, 24, 16, 23], X[16, 1, 9, 8], X[18, 33, 19, 34], 
   X[20, 13, 21, 12], X[22, 38, 23, 37], X[24, 2, 17, 1], 
   X[25, 39, 26, 40], X[27, 22, 28, 21], X[30, 9, 31, 10], 
   X[31, 17, 32, 18], X[34, 11, 35, 10], X[36, 28, 37, 29], 
   X[38, 14, 39, 15], X[40, 4, 33, 3]]

\item[L246] PD[X[2, 25, 3, 32], X[4, 13, 5, 12], X[6, 28, 7, 29], 
   X[7, 37, 8, 36], X[11, 33, 12, 34], X[13, 22, 14, 21], 
   X[15, 38, 16, 37], X[16, 1, 9, 8], X[17, 10, 18, 9], 
   X[20, 5, 21, 6], X[22, 26, 23, 27], X[24, 2, 17, 1], 
   X[25, 39, 26, 40], X[27, 15, 28, 14], X[30, 18, 31, 19], 
   X[31, 10, 32, 11], X[34, 20, 35, 19], X[35, 29, 36, 30], 
   X[38, 23, 39, 24], X[40, 4, 33, 3]]

\item[L249] = PD[X[2, 25, 3, 32], X[4, 13, 5, 12], X[6, 36, 7, 35], 
   X[7, 29, 8, 30], X[11, 33, 12, 34], X[13, 22, 14, 21], 
   X[15, 38, 16, 37], X[16, 1, 9, 8], X[17, 10, 18, 9], 
   X[20, 5, 21, 6], X[22, 26, 23, 27], X[24, 2, 17, 1], 
   X[25, 39, 26, 40], X[27, 15, 28, 14], X[30, 18, 31, 19], 
   X[31, 10, 32, 11], X[34, 20, 35, 19], X[36, 28, 37, 29], 
   X[38, 23, 39, 24], X[40, 4, 33, 3]].
\end{enumerate}

The process of testing each link against the Chen link and its mirror image and then converting back to the braid word is provided in GitHub \cite{Guo}.

\

By using Regina we verify that $L235 \sim L298$,  $L176\sim L246 \sim L249$, and $L270 \sim \text{Chen link}$ by a sequence of Reidemeister moves. We verify that $L235, L176,$ the Chen link, and their mirror images are not isotopic to each other by using the Jones polynomial. Recall that the Jones polynomial of the link distinguishes the link from its mirror image when the coefficients are not palindromic. 
\begin{eqnarray*}
    J_{L176}(t) &=& 232 - t^{-7} + 9t^{-6} - 36t^{-5} + 85t^{-4} - 144t^{-3} + 202t^{-2} - 232t^{-1} - 
 194 t + 140 t^2 - 67 t^3 + 18 t^4 \\
 &&+ 18 t^5 - 25 t^6 + 15 t^7 - 
 5 t^8 + t^9 \\
    J_{L235}(t) &=& t^{16}-8 t^{15}+27 t^{14}-51 t^{13}+66 t^{12}-63 t^{11}+46 t^{10}-9 t^9-18 t^8+52 t^7-66 t^6+64 t^5-43 t^4 \\
    &&+23 t^3-5 t^2\\
    J_{L270}(t) &=& 221 + 1t^{-6} - 8t^{-5} + 32t^{-4} - 78t^{-3} + 136t^{-2} - 186t^{-1} - 218 t + 
 193 t^2 - 140 t^3 + 81 t^4 - 26 t^5 \\
 &&+ 14 t^7 - 8 t^8 + 2 t^9.
\end{eqnarray*}

Below we provide in braid form the six pairwise non-isotopic $20$ crossing links that are counterexamples to the Montesinos-Nakanishi conjecture, all of which are $3$-move equivalent to the Chen link: 

\begin{equation}\label{eqn1}
\text{The Chen link: } \sigma^{-1}_{1} \sigma_{ 2} \sigma^{-1}_{1} \sigma_{ 3} \sigma_{ 2} \sigma^{-1}_{4} \sigma_{ 3} \sigma_{ 2} \sigma^{-1}_{1} \sigma_{ 2} \sigma^{-1}_{4} \sigma_{ 3} \sigma_{ 2} \sigma^{-1}_{1} \sigma_{ 2} \sigma^{-1}_{4} \sigma_{ 3} \sigma_{ 2} \sigma^{-1}_{4} \sigma_{ 3}
\end{equation}
\begin{equation}
\text{Chen's mirror image: }\sigma_{1} \sigma^{-1}_{2} \sigma_{ 1} \sigma^{-1}_{3} \sigma^{-1}_{2} \sigma_{ 4} \sigma^{-1}_{3} \sigma^{-1}_{2} \sigma_{ 1} \sigma^{-1}_{2} \sigma_{ 4} \sigma^{-1}_{3} \sigma^{-1}_{2} \sigma_{ 1} \sigma^{-1}_{2} \sigma_{ 4} \sigma^{-1}_{3} \sigma^{-1}_{2} \sigma_{ 4} \sigma^{-1}_{3}
\end{equation}
\begin{equation}
\sigma_{1} \sigma_{ 2} \sigma^{-1}_{3} \sigma_{ 2} \sigma_{ 1} \sigma_{ 4} \sigma_{ 3} \sigma^{-1}_{2} \sigma_{ 1} \sigma_{ 3} \sigma^{-1}_{4} \sigma_{ 3} \sigma_{ 2} \sigma_{ 1} \sigma_{ 3} \sigma^{-1}_{4} \sigma_{ 3} \sigma^{-1}_{2} \sigma_{ 3} \sigma_{ 4} 
 \end{equation}
\begin{equation}
\sigma^{-1}_{1} \sigma^{-1}_{2} \sigma_{ 3} \sigma^{-1}_{2} \sigma^{-1}_{1} \sigma^{-1}_{4} \sigma^{-1}_{3} \sigma_{ 2} \sigma^{-1}_{1} \sigma^{-1}_{3} \sigma_{ 4} \sigma^{-1}_{3} \sigma^{-1}_{2} \sigma^{-1}_{1} \sigma^{-1}_{3} \sigma_{ 4} \sigma^{-1}_{3} \sigma_{ 2} \sigma^{-1}_{3} \sigma^{-1}_{4}
\end{equation}
\begin{equation}
\sigma^{-1}_{1} \sigma_{ 2} \sigma_{ 3} \sigma_{ 2} \sigma^{-1}_{4} \sigma^{-1}_{3} \sigma_{ 2} \sigma^{-1}_{1} \sigma_{ 2} \sigma^{-1}_{3} \sigma_{ 2} \sigma_{ 4} \sigma^{-1}_{3} \sigma_{ 2} \sigma^{-1}_{1} \sigma_{ 2} \sigma_{ 4} \sigma^{-1}_{3} \sigma_{ 2} \sigma^{-1}_{4}
\end{equation}
\begin{equation}\label{eqn6}
\sigma_{1} \sigma^{-1}_{2} \sigma^{-1}_{3} \sigma^{-1}_{2} \sigma_{ 4} \sigma_{ 3} \sigma^{-1}_{2} \sigma_{ 1} \sigma^{-1}_{2} \sigma_{ 3} \sigma^{-1}_{2} \sigma^{-1}_{4} \sigma_{ 3} \sigma^{-1}_{2} \sigma_{ 1} \sigma^{-1}_{2} \sigma^{-1}_{4} \sigma_{ 3} \sigma^{-1}_{2} \sigma_{ 4}.
\end{equation}
    
\section{Burnside group of a link}
We use Burnside groups of links to show that the Chen link is not 3-move equivalent to a trivial link.
For every link we associate a group, which we call the $n$th Burnside group of the link, $\mathcal{B}_n(L)$, and prove that the group is preserved by $n$-moves. For $n=3$, the group is always finite and can be used to find links that are not $3$-move equivalent to a trivial link \cite{DP1,DP2}. 
One succinct definition of $\mathcal{B}_n(L)$ is $\pi_1(M^{(2)}_L)/(w^n)$, where $\pi_1(M^{(2)}_L)$ is the homology of the double branched cover of $S^3$ branched along $L$, however for practical applications we need a workable, diagrammatic definition of $\mathcal{B}_n(L)$.
We will do this in two steps, first discussing the classical Burnside groups and then the core group of a link via a link diagram.

\subsection{Classical Burnside groups}\label{burnsidesection}\ 
Let us first recall the famous Burnside conjecture concerning groups $B(r,n)=F_r/(w^n)$, where $F_r$ is the free group on $r$ generators and $w$ is an arbitrary element of $F_r$. 
\begin{problem}\label{Burnside} (Burnside 1902 \cite{Bu}) For what
values of $r$ and $n$ is the Burnside group $B(r,n)$ finite$?$
\end{problem}
In \cite{Bu} Burnside proved that for $n=3$ the group $B(r,3)$ is finite. Furthermore, Levi and van der Waerden \cite{LW} proved that the group $B(r,3)$ has $3^{n+ {\binom{n}{2}}+ \binom{n}{3}}$ elements.
For the actual status of the Burnside problem see \cite{Kos,V-Lee}.\footnote{It is finite for $r=2$, and for $n=2,3,4,6$ with arbitrary $r$. $B(r,n)$ is infinite for $r\geq 2$, and $n \geq 665$ with $r$ odd \cite{NA,Adj}. $B(2,2^k)$ was found to be infinite for $k\geq 13$ \cite{Lys}. We know that $B(2,5)$ is either infinite or has $5^{34}$ elements \cite{V-Lee}.}
\subsection{The core group of a link}
The $n$th Burnside group of a link, $\mathcal{B}_n(D)$, is the quotient of the (reduced) core group of the link.  
The core group of the link (denoted Cor$(D)$) is defined on a diagram D
 as the group generated by arcs of $D$ with the relation for every crossing $ba^{-1}bc^{-1}=1$ where $b$ is the overcrossing arc and 
$a$ and $c$ are undercrossings (see Figure \ref{CoreGroup}).
Observe that if we abelianize $\text{Cor}(D)$ we obtain the 
(universal) group of Fox colorings ($2b-a-c=0$).
If we arbitrarily assign the identity element, $1$, to one arc we obtain the reduced 
Core group $\text{Cor}^{red}(D)$.
\begin{figure}[ht]
\centering
\includegraphics[scale=.4]{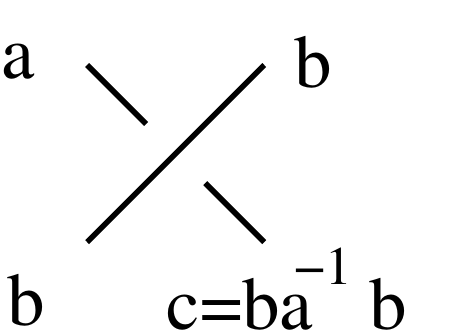}
\caption{Core group relation at the crossing: $ba^{-1}bc^{-1}=1.$}
\label{CoreGroup}
\end{figure}
\begin{theorem}[\cite{P-3, Wad}]
$\text{Cor}^{red}(D)= \pi_1(M^{(2)}_D)$ (the fundamental group of the double branch cover of $S^3$ with branching set $D$).
\end{theorem}
If we apply an $n$-move to our diagram then arcs $(b,a)$ go to $((ba^{-1})^nb,(ba^{-1})^na)$ and we observe that if  $(ba^{-1})^n=1$ then the group is unchanged (this is illustrated in Figure \ref{Burnside-n-move}).
This was motivation in \cite{DP1,DP2} to introduce the concept of Burnside group of links.
\begin{figure}[H]
\centering
\includegraphics[scale=.23]{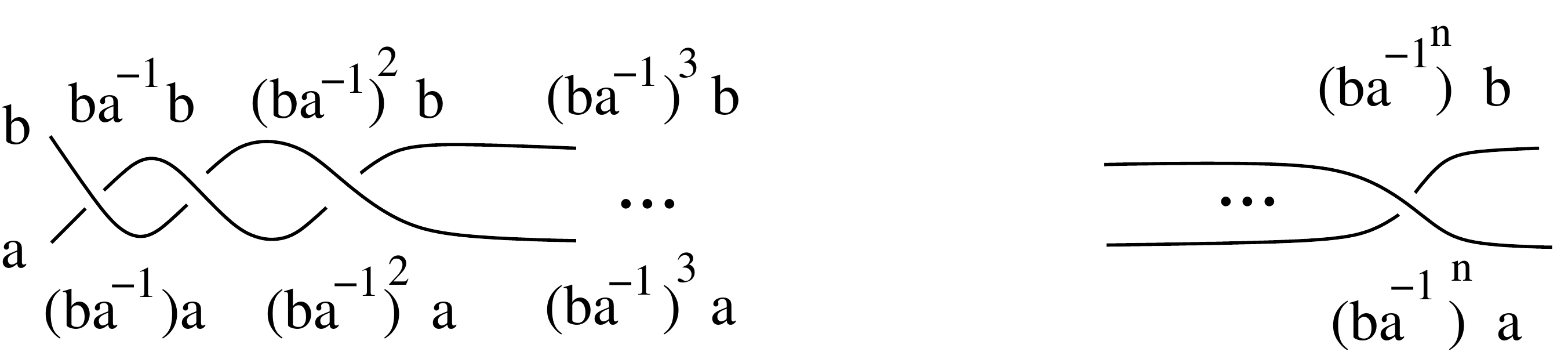}
\caption{$n$-move changes $(a,b)$ to $(ba^{-1})^na,(ba^{-1})^nb$.}
\label{Burnside-n-move}
\end{figure}
\subsection{The Burnside group of a link from the core group}\ 
The $n$th Burnside group of $D$, $\mathcal{B}_n(D)$, is the quotient of $\text{Cor}^{red}(D)$ by the subgroup  generated by all elements of type $w^n$ where $w$ is an arbitrary element of  
$\text{Cor}^{red}(D)$ (not only a generator). For us the classical result of 
Burnside and Levi-van der Waerden, mentioned before is important because $F_r/(w^3)$ is a finite group of $3^{n+ {\binom{n}{2}}+ {\binom{n}{3}}}$ elements, where $F_r$ is the free group of $r$ generators.
In particular, $|F_4/(w^3)|= 3^{4+6+4}=3^{14}=4,782,969$. Notice that the classical $n$th Burnside group 
is the $n$th Burnside group of the trivial link of $r+1$ components.

\subsection{5-braids, the 3-move operation, and the third Burnside group}
For all the remaining 493 links up to 20 crossings discussed in Section \ref{section:computation}, there exists an orientation such that the oriented link has $5$ Seifert circles. By the Yamada theorem, this implies that they all have $5$-braid representatives (in the conversion one uses Vogel- Traczyk algorithm \cite{Tra,Vo}). By Coxeter's result in \cite{Cox}, the group $C_5=B_5/(\sigma_1^3)$ is finite and has 102 conjugacy classes. Qi Chen proved that any closed 5-braid is 3-move equivalent to a trivial link or the 5-braid of Figure \ref{ChenLink}, which we call the Chen link. In particular, if a closed 5-braid is congujate to the Chen link or its mirror image, then it is $3$-move equivalent to the Chen link, otherwise it is 3-move equivalent to a trivial link.
The third Burnside  group of a link allows us to conclude that the Chen link is not $3$-move equivalent to any trivial link (its third Burnside group has $3^{10}$ elements, unlike any trivial link, \cite{DP1}). In Section \ref{section:computation}, every link up to $20$ crossings are 3-move equivalent to a trivial link or the Chen link therefore confirming Theorem \ref{Main}.

\section{Future directions}\label{Future} 
Cubic skein modules are defined using a cubic skein relation, which is a deformation of a $3$-move and has the form $b_3D_3+b_2D_2+b_1D_1+b_0D_0+ b_\infty D_\infty =0$. Thus, 
for $b_3=-b_0=1$, $b_2=b_1=b_\infty =0$ this skein relation describes $3$-move equivalence. An in-depth lecture on skein modules is given in \cite{PBIMW} and the fundamentals of cubic skein modules are discussed in \cite{BCGIMMPW}. We also refer the reader to \cite{DJP,P-5,P-6,P-7,P-Ts}. It is now a natural question whether our main result generalizes to cubic skein modules, in particular, whether any link in $S^3$ of 19 or less crossings is a linear combination of trivial links when $b_3$ and $b_0$ are invertible.  

\begin{conjecture}
    In the cubic skein module of $S^3$, with $b_0$ and $b_3$ invertible, every link up to 19 crossings is a linear combination of trivial links. Every link up to 20 crossings is a linear combination of trivial links and the Chen link.  
\end{conjecture}
In 1979, before the 3-move conjecture was formally stated, Nakanishi formulated the $4$-move conjecture\footnote{See \cite{P-1} for a discussion about $n$-moves.}:
\begin{conjecture}[Nakanishi]\cite{Kir}
Every knot can be reduced to the trivial knot by $4$-moves.
\end{conjecture}

The conjecture was proven for knots up to 11 crossings 
and the case of 12 crossings is discussed in \cite{P-Jab}. Together with Askitas, the seventh author guessed that the conjecture does not hold for the $2$-cable of the figure-eight knot with a half-twist (a $17$ crossing knot). Our method has the potential to give a positive answer to the conjecture for knots up to 16 crossings, if not more.

\section*{Acknowledgements}
This project started at MATRIX, a research institute in Creswick, Australia affiliated to the University of Melbourne, during the program Low Dimensional Topology: Invariants of Links, Homology Theories, and Complexity in June 2024. The first author was supported by Dr. Max Rössler, the Walter Haefner Foundation, and the ETH Zürich Foundation while she was a Fellow at the Institute for Theoretical Studies at ETH Zürich. The fourth author acknowledges the support by the Australian Research Council grant DP240102350. The fifth author was supported by the Matrix-Simons travel grant and acknowledges the support of the National Science Foundation through Grant DMS-2212736. The sixth author was supported by the Matrix-Simons travel grant and is grateful to P. Vojt\v{e}chovsk\'{y} for his support through the Simons Foundation Mathematics and Physical Sciences Collaboration Grant for Mathematicians no. 855097. The seventh author was partially supported by the Simons Collaboration Grant 637794.

\end{document}